\documentclass[12pt,reqno,a4paper]{amsart}
\usepackage{amsmath,amssymb,amsthm}
\usepackage{enumitem}
\usepackage[margin=2.5cm,top=3.5cm,bottom=3.5cm,footskip=1cm,headsep=1cm]{geometry}
\usepackage{amsaddr}
\usepackage[utf8]{inputenc}
\usepackage{graphicx}

\author{H. Egger$^*$ \and T. Kugler$^*$ \and N. Strogies$^\dag$}
\address{$^*$Department of Mathematics, TU Darmstadt, Germany\\%
$^\dag$Weierstrass Institut for Applied Analysis and
Stochastics, Berlin, Germany}
\email{egger@mathematik.tu-darmstadt.de}
\email{kugler@mathematik.tu-darmstadt.de}
\email{Nikolai.Strogies@wias-berlin.de}

\title[Parameter identification in a semilinear hyperbolic system]{Parameter identification \\in a semilinear hyperbolic system}

\newtheorem{lemma}{Lemma}[section]

\newtheorem{theorem}[lemma]{Theorem}

\theoremstyle{definition}
\newtheorem{remark}[lemma]{Remark}
\newtheorem*{example*}{Example}

\def\dt{\partial_t}
\def\dx{\partial_x}
\def\dtt{\partial_{tt}}
\def\dxx{\partial_{xx}}

\def\ddt{\tfrac{d}{dt}}

\def\umin{\underline{u}}
\def\umax{\overline{u}}

\def\tp{\triangle p}
\def\tpb{\triangle \bar p}

\def\RR{\mathbb{R}}

\def\eps{\varepsilon}

\numberwithin{equation}{section}
\numberwithin{table}{section}
\numberwithin{figure}{section}

\begin{document}

\begin{abstract} 
We consider the identification of a nonlinear friction law in a one-dimensional damped wave equation from additional boundary measurements.
Well-posedness of the governing semilinear hyperbolic system is established via semigroup theory and contraction arguments. 
We then investigte the inverse problem of recovering the unknown nonlinear damping law from additional boundary measurements of the pressure drop along the pipe. This coefficient inverse problem is shown to be ill-posed and a variational regularization method is considered for its stable solution. We prove existence of minimizers for the Tikhonov functional and discuss the convergence of the regularized solutions under an approximate source condition.
The meaning of this condition and some arguments for its validity are discussed 
in detail and numerical results are presented for illustration of the theoretical findings.
\end{abstract}

\maketitle

\vspace*{-1em}

\begin{quote}
\noindent 
{\small {\bf Keywords:} 
parameter identification,
semilinear wave equation, 
nonlinear inverse problem, 
Tikhonov regularization,
optimal control}

\end{quote}

\begin{quote}
\noindent
{\small {\bf AMS-classification (2000):}
35R30, 
49J20, 
49N45, 
65J22, 
74J25 
} 
\end{quote}

\section{Introduction} \label{sec:intro}

%
%
We consider a one-dimensional semilinear hyperbolic system of the form
\begin{align}
            \dt p(x,t) + \dx u(x,t) &= 0, \qquad x \in (0,1), \ t>0, \label{eq:sys1} \\
\dt u(x,t) + \dx p(x,t) + a(u(x,t)) &= 0, \qquad x \in (0,1), \ t>0, \label{eq:sys2}
\end{align}
which models, for instance, the damped vibration of a string or the propagation of pressure waves in a gas pipeline.
In this latter application, which we consider as our model problem in the sequel, 
$p$ denotes the pressure, $u$ the velocity or mass flux, and 
the nonlinear damping term $a(u)$ accounts for the friction at the pipe walls. 
The two equations describe the conservation of mass and the balance of momentum
and they can be obtained under some simplifying assumptions from the one dimensional Euler equations with friction; see e.g. \cite{BrouwerGasserHerty11,Guinot08,LandauLifshitz6}.
Similar problems also arise as models for the vibrations of elastic multistructures \cite{LagneseLeugeringSchmidt} 
or in the propagation of heat waves on microscopic scales \cite{JosephPreziosi89}.

The system \eqref{eq:sys1}--\eqref{eq:sys2} is complemented by boundary conditions 
\begin{align}
   u(0,t) = g_0(t), \quad u(1,t) = g_1(t), \qquad  t>0, \label{eq:sys3}
\end{align}
and we assume the initial values to be known and given by
\begin{align}
  p(x,0) = p_0(x), \quad u(x,0) = u_0(x), \qquad x \in (0,1). \label{eq:sys4}
\end{align}
Motivated by well-known friction laws in fluid dynamics for pipes \cite{LandauLifshitz6}, 
we will assume here that there exist positive constants $a_1,a_2$ such that 
\begin{align} \label{eq:a1} 
0 < a_1 \le a'(\xi) \le a_2 \qquad \text{for all } \xi \in \RR.   
\end{align}
In particular, friction forces are monotonically increasing with velocity.  
This condition allows us to establish well-posedness of the system \eqref{eq:sys1}--\eqref{eq:sys4}.
It is also reasonable to assume that $a(-\xi)=-a(\xi)$, i.e., 
the magnitude of the friction force does not depend on the flow direction,
and consequently we will additionally assume that $a(0)=0$.

\medskip

In this paper, we are interested in the inverse problem of determining an unknown 
friction law $a(u)$ in \eqref{eq:sys1}--\eqref{eq:sys4} from additional observation 
of the pressure drop
\begin{align}
  h(t)=\tp(t):= p(0,t) - p(1,t) , \qquad t > 0 \label{eq:sys5}
\end{align}
along the pipe. Such data are readily available in applications, e.g., gas pipeline networks. 

\medskip 

Before proceeding, let us comment on previous work for related coefficient inverse problems. 
By combination of the two equations \eqref{eq:sys1}--\eqref{eq:sys2}, one obtains 
the second order form 
\begin{align} \label{eq:second}
\dtt u - \dxx u + a'(u) \dt u = 0, \qquad x \in (0,1), \ t>0,
\end{align}
of a wave equation with nonlinear damping. 
The corresponding linear problem with coefficent $a'(u)$ replaced by $c(x)$ has been considered in \cite{Baudouin13,Bukgheim01,ImanuvilovYamamoto01}; 
uniqueness and Lipschitz stability for the inverse coefficient problem 
have been established in one and multiple space dimensions.
%
%
A one-dimensional wave equation with nonlinear source term $c(u)$ instead of $a'(u) \dt u$ 
has been investigated in \cite{CannonDuChateau83}; 
numerical procedures for the identification and some comments on the identifiability 
have been provided there.
In \cite{Kaltenbacher07}, the identification of the parameter function $c(\dx u)$ 
in the quasilinear wave equation $\dtt u - \dx ( c(\dx u) \dx u) = 0$
has been addressed in the context of piezo-electricity;
uniqueness and stability has been established for this inverse problem. 
Several results are available for related coefficient inverse problems for 
nonlinear parabolic equations; see e.g. 
\cite{CannonDuChateau73,DuChateau81,EggerEnglKlibanov05,EggerPietschmannSchlottbom15,Isakov93,Lorenzi86}.
Let us also refer to \cite{Isakov06,KlibanovTimonov} for an overview of available results and further references. 

To the best of our knowledge, the uniqueness and stability for the nonlinear coefficient problem
\eqref{eq:sys1}--\eqref{eq:sys5} considered in this paper are still open.
Following arguments proposed in \cite{EggerPietschmannSchlottbom15} for the analysis of 
a nonlinear inverse problem in heat conduction, we will derive \emph{approximate stability 
results} for the inverse problem stated above, which can be obtained by comparison with 
the linear inverse problem for the corresponding stationary problem. 
This allows us to obtain quantitative estimates for the reconstruction errors in dependence of the experimental setup,
and provides a theoretical basis for the hypothesis that uniqueness holds, 
if the boundary fluxes $g_i(t)$ are chosen appropriately. 

For the stable numerical solution in the presence of measurement errors,
we consider a variational regularization defined by
\begin{align}
J(a;p,u) &= \int_0^T |\tp(t) - h^\delta(t)|^2 dt + \alpha \|a-a^*\|^2 \to \min    \label{eq:min1}\\
         & \text{subject to } \quad \eqref{eq:sys1}-\eqref{eq:sys4} \quad \text{and} \quad \eqref{eq:a1}. \label{eq:min2}
\end{align}
This allows us to appropriately address the ill-posedness of the inverse coefficient problem \eqref{eq:sys1}--\eqref{eq:sys5}.
Here $\alpha>0$ is the regularization parameter, 
$a^*\in \RR$ is an a-priori guess for the damping law, and $h^\delta$ denotes the measurements of the pressure drop 
across the pipe for the time interval $[0,T]$. The precise form of regularization term will be specified below.

\medskip 

As a first step, we establish the well-posedness  of the system \eqref{eq:sys1}--\eqref{eq:sys4}
and prove uniform a-priori estimates for the solution. Semigroup theory for semilinear evolution 
problems will be used for that.
In addition, we also show the continuous dependence and differentiability of the states $(u,p)$ with respect to the parameter $a(\cdot)$. 

We then study the optimal control problem \eqref{eq:min1}--\eqref{eq:min2}. 
Elimination of $(p,u)$ via solution of \eqref{eq:sys1}--\eqref{eq:sys4} 
leads to  reduced minimization problem corresponding to Tikhonov regularization for 
the nonlinear inverse problem $F(a)=h^\delta$ where $F$ is the parameter-to-measurment mapping 
defined implicitly via the differential equations. 
Continuity, compactness, and differentiability of the forward operator $F$ are investigated, 
providing a guidline for the appropriate functional analytic setting for the inverse problem.
The existence and stability of minimizers for \eqref{eq:min1}--\eqref{eq:min2} then follows with standard arguments. 
In addition, we derive quantitative estimates for the error between the reconstructed 
and the true damping parameter using an \emph{approximate source condition},
which is reasonable for the problem under consideration. 
Such conditions have been used successfully for the analysis of Tikhonov regularization 
and iterative regularization methods in \cite{EggerSchlottbom11,HeinHofmann05}. 

As a third step of our analysis, we discuss in detail the meaning and the plausibility of this approximate source condition. 
We do this by showing that the nonlinear inverse problem is actually close to a linear inverse 
problem, provided that the experimental setup is chosen appropriately. 
This allows us to derive an approximate stability estimate for the inverse problem, 
and to justify the validity of the approximate source condition. 
These results suggest the hypothesis of uniqueness for the inverse problem under investigation, 
and they allow us to make predictions about the results that can be expected in practice and that are actually observed in our numerical tests.

\medskip 

The remainder of the manuscript is organized as follows: 
In Section~\ref{sec:prelim} we fix our notation and briefly discuss the underlying linear wave equation without damping.
The well-posedness of the state system \eqref{eq:sys1}--\eqref{eq:sys4} is established in Section~\ref{sec:state}
via semigroup theory. For convenience of the reader, some auxiliary results are summarized in an appendix.
In Section~\ref{sec:forward}, we then investigate the basic properties of the parameter-to-measurement mapping $F$. 
Section~\ref{sec:min} is devoted to the analysis of the regularization method \eqref{eq:min1}--\eqref{eq:min2}
and provides a quantitative estimate for the reconstruction error.
The required approximate source condition and the approximate stability of the inverse problem are discussed in Section~\ref{sec:hyp}
in detail. 
Section~\ref{sec:num} presents the setup and the results of our numerical tests.
We close with a short summary of our results and a discussion of possible directions for future research.

\section{Preliminaries} \label{sec:prelim}

Throughout the manuscript, we use standard notation for Lebesgue and Sobolev spaces and for classical functions spaces, see e.g. \cite{Evans98}. For the analysis of problem \eqref{eq:sys1}--\eqref{eq:sys4},
we will employ semigroup theory. 
The evolution of this semilinear hyperbolic system is driven by the linear wave equation
\begin{align}
\dt p(x,t) + \dx u(x,t) &= 0, \quad x \in (0,1), \ t>0, \\   
\dt u(x,t) + \dx p(x,t) &= 0, \quad x \in (0,1), \ t>0,
\end{align}
with homogeneous boundary values 
\begin{align}
u(0,t)=u(1,t)=0, \quad t>0,
\end{align}
and initial conditions given by $p(\cdot,0)=p_0$ and $u(\cdot,0)=u_0$ on $(0,1)$.
This problem can be written in compact form as an abstract evolution equation
\begin{align} \label{eq:abstract}
y'(t) + A y(t) = 0, \ t>0, \qquad y(0)=y_0,
\end{align}
with state vector $y=(p,u)$, initial value $y_0=(p_0,u_0)$, and operator $A=\begin{pmatrix} 0 & \dx \\ \dx & 0\end{pmatrix}$. \\[-1ex]
The starting point for our analysis is the following
\begin{lemma}[Generator] \label{lem:generator} 
Let $X=L^2(0,1) \times L^2(0,1)$ and $D(A)=H^1(0,1) \times H_0^1(0,1)$. \\
Then the operator $A : D(A) \subset X \to X$ generates a $C^0$-semigroup of contractions on $X$.
\end{lemma}
\begin{proof}
One easily verifies that $A$ is a densly defined and closed linear operator on $X$. 
Moreover, $(A y, y)_X = 0$ for all $y \in D(A)$; therefore, $A$ is dissipative. 
By direct calculations, one can see that for any $\bar f, \bar g \in L^2(0,1)$, the boundary value problem\begin{align*}
\bar p(x) + \dx \bar u(x) &= \bar f(x), \quad x \in (0,1),\\ 
\bar u(x) + \dx \bar p(x) &= \bar g(x), \quad x \in (0,1),
\end{align*}
with $\bar u(0)=\bar u(1)=0$ is uniquely solvable with solution $(\bar p,\bar u) \in H^1(0,1) \times H_0^1(0,1)$.
The assertion hence follows by the Lumer-Phillips theorem \cite[Ch~1, Thm~4.3]{Pazy83}.  
\end{proof}
The analysis of the model problem \eqref{eq:sys1}--\eqref{eq:sys4} can now be done in the framework of semigroups. 
For convenience, we collect some of the required results in the appendix.

\section{The state system} \label{sec:state}

Let us return to the semilinear wave equation under consideration.
For proving well-posedness of the system \eqref{eq:sys1}--\eqref{eq:sys4},
and in order to establish some additional regularity of the solution, we will assume that 
\begin{itemize}\setlength\itemsep{1ex}
 \item[(A1)] $a \in W_{loc}^{3,\infty}(\RR)$ with $a(0)=0$, $a_0 \le a'(\cdot) \le a_1$, $|a''(\cdot)| \le a_2$, and $|a'''(\cdot)| \le a_3$
\end{itemize}
for some positive constants $a_0,a_1,a_2,a_3>0$. 
%
Since the damping law comes from a modelling process involving several approximation steps, 
these assumptions are not very restrictive in practice.
In addition, we require the initial and boundary data to satisfy 
\begin{itemize}\setlength\itemsep{1ex}
 \item[(A2)] $u_0=0$ and $p_0=c$ with $c \in \RR$;
 \item[(A3)] $g_0,g_1 \in C^4([0,T])$ for some $T>0$, $g_0(0)=g_1(0)=0$, and $g_0'(0)=g_1'(0)=0$. 
\end{itemize}
The system thus describes the smooth departure from a system at rest.
As will be clear from our proofs, 
the assumptions on the initial conditions and the regularity requirements for the parameter and the initial and boundary data 
could be relaxed without much difficulty.
Existence of a unique solution can now be established as follows.

\begin{theorem}[Classical solution] \label{thm:classical} $ $\\
Let (A1)--(A3) hold. 
Then there exists a unique classical solution 
\begin{align*} 
(p,u) \in C^1([0,T];L^2(0,1) \times L^2(0,1)) \cap C([0,T]; H^1(0,1) \times H^1(0,1))  
\end{align*}
for the initial boundary value problem \eqref{eq:sys1}--\eqref{eq:sys4} and its norm can be bounded by 
\begin{align*}
\|(p,u)\|_{C([0,T];H^1\times H^1)} + \|(p,u)\|_{C^1([0,T];L^2 \times L^2)} 
\le C'
\end{align*}
with constant $C'$ only depending on the bounds for the coefficients and the data and the time horizon $T$.  
Moreover, $\tp := p(0,\cdot)-p(1,\cdot) \in C^\gamma([0,T])$, for any $0 \le \gamma < 1/2$, and
\begin{align*}
\|\tp\|_{C^{\gamma}([0,T])} \le C'(\gamma).
\end{align*}
\end{theorem}
\begin{proof}
The proof follows via semigroup theory for semilinear problems \cite{Pazy83}. 
For convenience of the reader and to keep track of the constants, we sketch the basics steps:

{\em Step 1:}
We define $\hat u(x,t) = (1-x) g_0(t) + x g_1(t)$ and set 
$\hat p(x,t) = \int_0^x \hat p_x(s,t) dx$ with $\hat p_x(x,t) = (x-1) (a(g_0(t))+g_0'(t))  -x (a(g_1(t))+g_1'(t))$.
Then we decompose the solution into $(p,u)=(\hat p,\hat u) + (\tilde p,\tilde u)$
and note that $(\hat p,\hat u) \in C^1([0,T];H^1 \times H^1)$ by construction and assumption (A3). 
The second part $(\tilde p, \tilde u)$ solves
\begin{align*}
&&&&\dt \tilde p + \dx \tilde u &= f_1, && \tilde p(\cdot,0) = \tilde p_0,&&&&\\
&&&&\dt \tilde u + \dx \tilde p &= f_2, && \tilde u(\cdot,0) = \tilde u_0,&&&&
\end{align*}
with $f_1(t)=-\dt \hat p(t) - \dx \hat u(t)$, $f_2(t,\tilde u(t))=-\dt \hat u(t)-\dx \hat p(t)- a(\hat u(t) + \tilde u(t))$
and initial values $\tilde p_0 = p_0 - \hat p(0)$, $\tilde u_0 = u_0 - \hat u(0)$.
In addition, we have $\tilde u(0,t)=\tilde u(1,t)=0$ for $t>0$.
This problem can be written as 
abstract evolution 
\begin{align*}
y'(t) + A y(t) = f(t,y(t)), \qquad y(0)=y_0,  
\end{align*}
on $X=L^2 \times L^2$ with $y=(\tilde p,\tilde u)$, $f(t,y)=(f_1(t),f_2(t,y_2))$, and $D(A)=H^1 \times H_0^1$.

{\em Step 2:}
We now verify the conditions of Lemma~\ref{lem:classical2} stated in the appendix.
By assumptions (A2) and (A3) one can see that $y_0 = (\tilde p_0,\tilde u_0) \in H^1(0,1) \times H_0^1(0,1)$.
For every $y \in H^1(0,1) \times H_0^1(0,1)$, we further have $f(t,y) = (f_1(t),f_2(t,y_2)) \in H^1(0,1) \times H_0^1(0,1)$ by construction of $\hat u$ and $\hat p$. Moreover, $f$ is continuous w.r.t. time. 
Denote by $|u|_{H^1}=\|\dx u\|_{L^2}$ the seminorm of $H^1$. Then 
\begin{align*}
|f_2(t,v) - f_2(t,w)|_{H^1} 
&= |a(\hat u(t) + v) - a(\hat u(t)+w)|_{H^1} \\
&= \int_0^1 |a'(\hat u(t) + (1-s) v + s w) (v-w)|_{H^1} ds \\
&\le a_1 |v-w|_{H^1} + a_2 |\hat u(t) + (1-s) v + s w|_{H^1} |v-w|_{H^1}.
\end{align*}
Here we used the embedding of $H^1(0,1)$ into $L^\infty(0,1)$ and the bounds for the coefficients. 
This shows that $f$ is locally Lipschitz continuous with respect to $y$, uniform on $[0,T]$.
By Lemma~\ref{lem:classical2}, we thus obtain local existence and uniqueness of a classical solution. 

{\em Step 3:}
To obtain the global existence of the classical solution, 
note that 
\begin{align*}
\|\ddt f(t,y(t))\|_X 
&\le \|\ddt f_1(t)\|_{L^2} + \|\ddt f_2(t,\tilde u(t))\|_{L^2}  \\
&\le C_1 + C_2 + \|\dt \tilde u(t)\|_{L^2}),
\end{align*}
where the first term comes from estimating $f_1$ and the other three terms from the estimate for $f_2$.
The constants $C_1,C_2$ here depend only on the bounds for the data.
Global existence of the classical solution and the uniform bound now follow from Lemma~\ref{lem:classical3}.
\end{proof}

Note that not all regularity assumptions for the data and for the parameter were required so far.
The conditions stated in (A1)--(A3) allow us to prove higher regularity of the solution,
which will be used for instance in the proof of Theorem~\ref{thm:lipschitz} later on.
\begin{theorem}[Regularity] \label{thm:regularity}
Under the assumptions of the previous theorem, we have 
\begin{align*}
\|(p,u)\|_{C^1([0,T];H^1\times H^1) \cap C^2([0,T];L^2 \times L^2)} \le C''
\end{align*}
with $C''$ only depending on the bounds for the coefficient and data, and the time horizon.
\end{theorem}
\begin{proof}
To keep track of the regularity requirements, we again sketch the main steps: 

{\em Step 1:}
Define $(r,w)=(\dt p,\dt u)$ and $(r,w) = (\hat r,\hat w) + (\tilde r,\tilde w)$ with $(\hat r,\hat w)=(\dt \hat p,\dt \hat u)$ and $(\tilde r,\tilde w)=(\dt \tilde p,\dt \tilde u)$ as in the previous proof.
The part $z=(\tilde r,\tilde w)$ can be seen to satisfy
\begin{align} \label{eq:z}
\dt z(t) + A z(t) = g(t,z(t)), \qquad z(0)=z_0,
\end{align}
with right hand side  $g(t,z)=(-\dt \hat r(t) -\dx \hat w(t),-\dt \hat w(t)-\dx \hat r(t)-a'(u(t)) z_2)$ 
and initial value $z_0=(\dt p(0)-\dt \hat p(0),\dt u(0)-\dt \hat u(0)) = (-\dx u_0-\dt \hat p(0),-\dx p_0 - a(u_0) - \dt \hat u(0))$.

{\em Step 2:}
Using the assumptions (A1)--(A3) for the coefficient and the data, and the bounds for the solution 
of Theorem~\ref{thm:classical}, and the definition of $\hat p$ and $\hat u$, 
one can see that $z_0 \in Y=H^1 \times H_0^1$ and that 
$g : [0,T] \times H^1(0,1) \times Y \to Y$ satisfies the conditions of Lemma~\ref{lem:classical2}. 
Thus $z(t)$ is a local classical solution. 

{\em Step 3:}
Similarly as in the previous proof, one can show that\begin{align*}
\|\ddt g(t,z(t))\|_X \le C_1 + C_2 \|z'(t)\|_X + C_3 \|A z(t)\|_X 
\end{align*}
for all sufficiently smooth function $z$. 
The global existence and uniform bounds for the classical solution then follow again by Lemma~\ref{lem:classical3}.
\end{proof}

\section{The parameter-to-output mapping} \label{sec:forward}

Let $u_0,p_0,g_0,g_1$ be fixed and satisfy assumptions (A2)--(A3).
Then by Theorem~\ref{thm:classical}, we can associate to any damping parameter $a$ 
satisfying the conditions (A1) the corresponding solution $(p,u)$ of problem 
\eqref{eq:sys1}--\eqref{eq:sys4}. 
By the uniform bounds of Theorem~\ref{thm:classical} and the embedding of $H^1(0,1)$ in $C[0,1]$,
we know that 
\begin{align} \label{eq:bounds}
\umin \le u(x,t) \le \umax, \qquad x \in \omega, \ 0 \le t \le T, 
\end{align}
for some constants $\umin$, $\umax$ independent of the choice of $a$. 
Without loss of generality, we may thus restrict the parameter function $a$ to the interval $[\umin,\umax]$.
We now define the parameter-to-measurement mapping, in the sequel also called \emph{forward operator}, by 
\begin{align} \label{eq:forward}
F : D(F) \subset H^2(\umin,\umax) \to L^2(0,T),
\qquad a \mapsto \tp
\end{align}
where $\tp=p(0,\cdot)-p(1,\cdot)$ is the pressure drop across the pipe and $(p,u)$ is the solution of \eqref{eq:sys1}--\eqref{eq:sys4} for parameter $a$. As domain for the operator $F$, we choose 
\begin{align} \label{eq:domain}
D(F)=\{ a \in H^2(\umin,\umax) : (A1) \text{ holds}\},
\end{align}
which is a closed and convex subset of $H^2(\umin,\umax)$.
By Theorem~\ref{thm:classical}, the parameter-to-measurment mapping is well-defined on $D(F)$.
In the following, we establish several further properties of this operator, which will be required for our analysis later on.
\begin{theorem}[Lipschitz continuity] \label{thm:lipschitz}
The operator $F$ is Lipschitz continuous, i.e., 
\begin{align*}
\|F(a) - F(\tilde a)\|_{L^2(0,T)} \le C_L \|a-\tilde a\|_{H^2(\umin,\umax)}, \qquad \forall a, \tilde a \in D(F)
\end{align*}
with some uniform Lipschitz constant $C_L$ independent of the choice of $a$ and $\tilde a$.
\end{theorem}
\begin{proof}
Let $a,\tilde a \in D(F)$ and let $(p,u)$, $(\tilde p,\tilde u)$ denote the corresponding classical solutions of problem \eqref{eq:sys1}--\eqref{eq:sys4}.
Then the function $(r,w)$ defined by $r=\tilde p-p$, $w=\tilde u-u$  satisfies
\begin{align*}
\dt r + \dx w &= 0, \\
\dt w + \dx r &= a(u) - \tilde a (\tilde u) =:f_2,  
\end{align*}
with initial and boundary conditions $r(x,0)=w(x,0)=w(0,t)=w(1,t)=0$. 
By Theorem~\ref{thm:classical}, we know the existence of a unique classical solution $(r,w)$.
Moreover,
\begin{align*}
\|\ddt f_2\|_{L^2} 
&\le \|(a'(u)-a'(\tilde u)) \dt u\|_{L^2} + \|(a'(\tilde u) - \tilde a'(\tilde u)) \dt u\|_{L^2} 
+ \|\tilde a'(\tilde u) (\dt u - \dt \tilde u)\|_{L^2} \\
&\le a_2 \|w\|_{L^2} \|\dt u\|_{L^\infty} + \|a'-\tilde a'\|_{L^\infty} \|\dt u\|_{L^2} + a_1 \|\dt w\|_{L^2}
\end{align*}
Using the uniform bounds for $u$ provided by Theorem~\ref{thm:classical} and \ref{thm:regularity} and similar estimates as in the proof of Lemma~\ref{lem:classical3}, one obtains
$\|(r,w)\|_{C([0,T];H^1 \times H_0^1)} \le C \|a' - \tilde a'\|_{L^\infty}$ with $C$ only depending on the bounds for the coefficients and the data and on the time horizon. The assertion then follows by noting that $F(\tilde a) -F(a) = r(0,\cdot)-r(1,\cdot)$ and the continuous embedding of $H^1(0,1)$ in $L^\infty(0,1)$ and $H^2(\umin,\umax)$ to $W^{1,\infty}(\umin,\umax)$.
\end{proof}

By careful inspection of the proof of Theorem~\ref{thm:lipschitz}, we also obtain 
\begin{theorem}[Compactness] \label{thm:compact}
The operator $F$ maps sequences in $D(F)$  weakly converging in $H^2(\umin,\umax)$ to strongly convergent sequences in $L^2(0,T)$. In particular, $F$ is compact. 
\end{theorem}
\begin{proof}
The assertion follows from the estimates of the previous proof by noting that the embedding of $H^2(\umin,\umax)$ into $W^{1,\infty}(\umin,\umax)$ is compact. The forward operator is thus a composition of a continuous and a compact operator.
\end{proof}

As a next step, we consider the differentiability of the forward operator.
\begin{theorem}[Differentiability] \label{thm:differentiability} $ $\\
The operator $F$ is Frechet differentiable with Lipschitz continuous derivative, i.e., 
\begin{align*}
\|F'(a) - F'(\tilde a)\|_{H^2(\umin,\umax) \to L^2(0,T)} \le L \|a-\tilde a\|_{H^2(\umin,\umax)}
\qquad \text{for all } a,\tilde a \in D(F).
\end{align*}
\end{theorem}

\begin{proof}
Denote by $(p(a),u(a))$ the solution of \eqref{eq:sys1}--\eqref{eq:sys4} 
for parameter $a$ and let $(r,w)$ be the directional derivative of $(p(a),u(a))$ with respect to $a$ in direction $b$, defined by 
\begin{align} \label{eq:directional}
r = \lim_{s \to 0} \frac{1}{s} (p(a+sb)-p(a)) 
\quad \text{and} \quad  
w = \lim_{s \to 0} \frac{1}{s} (u(a+sb)-u(a)). 
\end{align}
Then $(r,w)$ is characterized by the {\em sensitivity system}
\begin{align}
\dt r + \dx w &= 0, \label{eq:sen1} \\
\dt w + \dx r &= -a'(u(a)) w - b(u(a))=:f_2\label{eq:sen2}
\end{align}
with homogeneous initial and boundary values
\begin{align}
   r(x,0) = w(x,0) = w(0,t) = w(1,t) &= 0. \label{eq:sen34}
\end{align}
The right hand side $f_2(t,w)=-a'(u(a;t)) w - b(u(a;t))$ can be shown to be continuously differentiable 
with respect to time, by using the previous results and (A1)--(A3). 
Hence by Lemma~\ref{lem:classical} there exists a unique classical solution $(r,w)$ to \eqref{eq:sen1}--\eqref{eq:sen34}.
Furthermore
\begin{align*}
\|\ddt f_2\|_{L^2}
&\le \|a''(u)\|_{L^\infty} \|\dt u\|_{L^\infty} \|w\|_{L^2} + \|a'(u)\|_{L^\infty} \|\dt w\|_{L^2} + \|b'(u)\|_{L^\infty} \|\dt u\|_{L^2}.
\end{align*}
By Lemma~\ref{lem:classical3} we thus obtain uniform bounds for $(w,r)$.
The directional differentiability of $(p(a),u(a))$ follows by verifying \eqref{eq:directional}, 
which is left to the reader. 
The function $(r,w)$ depends linearly and continuously on $b$ and continuously on $a$ which yields the continuous differentiability of $(p(a),u(a))$ with respect to the parameter $a$. 
The differentiability of the forward operator $F$ then follows by noting that $F'(a) b = r(0,\cdot)-r(1,\cdot)$. 
For the Lipschitz estimate, we repeat the argument of Theorem~\ref{thm:lipschitz}.  An additional derivative 
of the parameter $a$ is required for this last step.
\end{proof}

\section{The regularized inverse problem} \label{sec:min}

The results of the previous section allow us to rewrite the constrained minimization problem 
\eqref{eq:min1}--\eqref{eq:min2} in reduced form as \begin{align} \label{eq:tikhonov} 
J_\alpha^\delta(a) := \|F(a) - h^\delta\|_{L^2(0,T)}^2 + \alpha \|a-a^*\|_{H^2(\umin,\umax)}^2 \to \min_{a \in D(F)},
\end{align}
which amounts to Tikhonov regularization for the nonlinear inverse problem $F(a)=h^\delta$. 
As usual, we replaced the exact data $h$ by perturbed data $h^\delta$ to account for measurement errors.
Existence of a minimizer can now be established with standard arguments \cite{EnglHankeNeubauer96,EnglKunischNeubauer89}.
\begin{theorem}[Existence of minimizers]
Let (A2)--(A3) hold. 
Then for any $\alpha>0$ and any choice of data $h^\delta \in L^2(0,T)$, 
the problem \eqref{eq:tikhonov} has a minimizer $a_\alpha^\delta \in D(F)$.
\end{theorem}
\begin{proof}
The set $D(F)$ is closed, convex, and bounded. In addition, we have shown that $F$ is weakly continuous,
and hence the functional $J_\alpha^\delta$ is weakly lower semi-continuous.  
Existence of a solution then follows as in \cite[Thm.~10.1]{EnglHankeNeubauer96}.
\end{proof}
\begin{remark}
Weak continuity and thus existence of a minimizer can be shown without the bounds for the second and third derivative of the parameter in assumption (A1). 
\end{remark}
Let us assume that there exists a true parameter $a^\dag \in D(F)$ 
and denote by $h=F(a^\dag)$ the corresponding exact data.
The perturbed data $h^\delta$ are required to satisfy 
\begin{align} \label{eq:noise}
\|h^\delta - h\|_{L^2(0,T)} \le \delta
\end{align} 
with $\delta$ being the noise level.
These are the usual assumptions  for the inverse problem.
%
In order to simplify the following statements about convergence, we also assume 
for the moment that the solution of the inverse problem is unique, i.e., that
\begin{align} \label{eq:unique}
F(a) \ne F(a^\dag) \quad \text{for all } a \in D(F) \setminus \{a^\dag\}.
\end{align}
This assumption is only made for convenience here, 
but we also give some justification for its validity in the following section. 
Under this uniqueness assumption, we obtain the following result about the convergence of the regularized solutions; see \cite{EnglHankeNeubauer96,EnglKunischNeubauer89}.
\begin{theorem}[Convergence]
Let \eqref{eq:unique} hold and $h^\delta$ be a sequence of data satisfying \eqref{eq:noise} for $\delta \to 0$.
Further, let $a_\alpha^\delta$ be corresponding minimizeres of \eqref{eq:tikhonov} with $\alpha=\alpha(\delta)$ chosen such that $\alpha \to 0$ and $\delta^2/\alpha \to 0$. 
Then $\|a_\alpha^\delta - a^\dag\|_{H^2(\umin,\umax)} \to 0$ with $\delta \to 0$.
\end{theorem}
\begin{remark}
Without assumption \eqref{eq:unique} about uniqueness, convergence holds for subsequences and 
towards an $a^*$-minimum norm solution $a^\dag$; see \cite[Sec.~10]{EnglHankeNeubauer96} for details. 
\end{remark}
To obtain quantitative estimates for the convergence, some additional conditions on the nonlinearity of the operator $F$ and on the solution $a^\dag$ are required. 
Let us assume that 
\begin{align} \label{eq:source}
a^\dag - a^* = F'(a^\dag)^* w + e 
\end{align}
holds for some $w \in L^2(0,T)$ and $e \in H^2(\umin,\umax)$. 
Note that one can always choose $w=0$ and $e=a^\dag-a^*$, so this condition is no 
restrictio of generalty. However, good bounds for $\|w\|$ and $\|e\|$ are required 
in order to really take advantage of this splitting later on.
Assumption \eqref{eq:source} is called a \emph{approximate source condition}, and has been 
investigated for the convergence analysis of regularization methods for instance in \cite{EggerSchlottbom11,HeinHofmann05} 
By a slight modification of the proof of \cite[Thm~10.4]{EnglHankeNeubauer96}, one can obtain
\begin{theorem}[Convergence rates]
Let \eqref{eq:source} hold and let $L \|w\|_{H^2(\umin,\umax)} < 1$. Then
\begin{align}
\|a^\dag - a_\alpha^\delta\|_{H^2(\umin,\umax)} 
\le C \big( \delta^2/\alpha + \alpha \|w\|^2 + \delta \|w\|_{H^2(\umin,\umax)} +  \|e\|_{L^2(0,T)}^2).
\end{align}
The constant $C$ in the estimate only depends on the size of $L\|w\|_{H^2(\umin,\umax)}$.
\end{theorem}
\begin{proof}
Proceeding as in \cite{EnglHankeNeubauer96,EnglKunischNeubauer89}, one can see that 
\begin{align*}
\|F(a_\alpha^\delta) - h^\delta\|^2 + \alpha \|a_\alpha^\delta - a^\dag\|^2 
\le \delta^2 + 2\alpha (a^\dag-a_\alpha^\delta, a^\dag - a^* ). 
\end{align*}
Using the approximate source condition \eqref{eq:source}, the last term can be estimated by 
\begin{align*}
(a^\dag - a_\alpha^\delta, a^\dag - a^* ) 
&= (F'(a^\dag) (a^\dag -a_\alpha^\delta) , w) + (a^\dag-a_\alpha^\delta,e) \\
&\le \|F'(a^\dag) (a^\dag-a_\alpha^\delta)\| \|w\| + \|a^\dag-a_\alpha^\delta\|\|e\|.
\end{align*}
By elementary manipulations and the Lipschitz continuity of the derivative, one obtains
\begin{align*}
\|F'(a^\dag) (a^\dag - a_\alpha^\delta)\| 
&\le \|F(a_\alpha^\delta)-h^\delta\| + \delta + \tfrac{L}{2} \|a_\alpha^\delta - a^\dag\|^2.
\end{align*}
Using this in the previous estimates and applying Young inequalities leads to 
\begin{align*}
&\|F(a_\alpha^\delta) - h^\delta\|^2 + \alpha \|a_\alpha^\delta - a^\dag\|^2 \\
&\le \delta^2 + 2 \alpha^2 \|w\|^2 + 2 \alpha \delta \|w\| + C' \alpha \|e\|^2 
  + \frac{1}{2} \|F(a_\alpha^\delta)-h^\delta\|^2 + \alpha \|a_\alpha^\delta-a^\dag\|^2 (L \|w\| + \tfrac{1}{C'}).
\end{align*}
If $L\|w\|<1$, we can choose $C'$ sufficienlty large such that $L \|w\| + \tfrac{1}{C'} < 1$ and the last two terms can be absorbed in the left hand side, which yields the assertion.
\end{proof}

\begin{remark}
The bound of the previous theorem yields a quantitative estimate for the error.
If the source condition \eqref{eq:source} holds with $e=0$ and $L \|w\| < 1$,  
then for $\alpha \approx \delta$ one obtains $\|a_\alpha^\delta - a^\dag\| = O(\delta^{1/2})$,
which is the usual convergence rate result \cite[Thm.~10.4]{EnglHankeNeubauer96}.
The theorem however also yields estimates and a guidline for the choice of the regularization parameter in the general case. 
We refer to \cite{HeinHofmann05} for an extensive discussion of the approximate source condition \eqref{eq:source} and its relation to more standard conditions.
\end{remark}

\begin{remark}
If the deviation from the classical source condition is small, i.e., if \eqref{eq:source} holds with 
$\|e\| \approx \delta^{1/2}$ and $L \|w\| < 1$, 
then for $\alpha \approx \delta$ one still obtains the usual estimate $\|a_\alpha^\delta - a^\dag\| = O(\delta^{1/2})$. 
As we will illustrate in the next section, the assumption that $\|e\|$ is small is realistic in practice, 
if the experimental setup is chosen appropriately. 
The assumption that $\|e\|$ is sufficiently small in comparison to $\alpha$ also allows to show 
that the Tikhonov functional is locally convex around minimizers and to prove convergence of iterative schemes;
see \cite{EggerSchlottbom15,ItoJin15} for some recent results in this direction.
\end{remark}

Numerical methods for minimizing the Tikhonov functional usually require the application of the adjoint derivative operator. 
For later reference, let us therefore briefly give a concrete representation of the adjoint that can be used for the implementation.

\begin{lemma}
Let $\psi \in H^1(0,T)$ with $\psi(T)=0$ and let $(q,v)$ denote the solution of 
\begin{align}
\dt q + \dx v &= 0 \qquad \qquad x \in (0,1), \ t<T,  \label{eq:adj1}\\
\dt v + \dx q &= a'(u) v,  \quad \ x \in (0,1), \ t<T, \label{eq:adj2}
\end{align}
with terminal conditions $v(x,T)=q(x,T)=0$ and boundary conditions 
\begin{align}
v(0,t)=v(1,t)=\psi(t), \quad t<T.  \label{eq:adj3}
\end{align}
Then the action of the adjoint operator $\phi=F'(a)^* \psi$ is given by 
\begin{align} \label{eq:adj}
(\phi,b)_{H^2(\umin,\umax)} = \int_0^T (b(u), v)_{L^2(0,1)} dt , \qquad \forall b \in H^2(\umin,\umax).
\end{align}
\end{lemma}
\begin{proof}
By definition of the adjoint operator, we have 
\begin{align*}
(b,F'(a)^* \psi)_{H^2(\umin,\umax)} = (F'(a) b, \psi)_{L^2(0,T)}. 
\end{align*}
Using the characterization of the derivative via the solution $(r,w)$ of the sensitivity equation \eqref{eq:sen1}--\eqref{eq:sen2} and the definition of the adjoint state $(q,v)$ via \eqref{eq:adj1}--\eqref{eq:adj2}, we obtain 
\begin{align*}
&(F'(a) b, \psi)_{L^2(0,T)}
 = \int\nolimits_0^T r(0,t) v(0,t) - r(1,t) v(1,t) dt \\
&= \int\nolimits_0^T -(\dx r,v) - (r, \dx v) dt 
=\int\nolimits_0^T (\dt w+a'(u) w + b(u), v) + (r,\dt q) dt \\
&= \int\nolimits_0^T -(\dt v - a'(u) v,w) + (b(u),v) + (\dx w,q) dt 
=\int\nolimits_0^T (b(u), v) dt.
\end{align*}
For the individual steps we only used integration-by-parts and made use of the boundary and initial conditions.
This already yields the assertion.
\end{proof}


\begin{remark}
Existence of a unique solution $(q,v)$ of the adjoint system \eqref{eq:adj1}--\eqref{eq:adj2} 
with the homogeneous terminal condition $v(x,T)=q(x,T)=0$ and boundary condition $q(0,t)=q(1,t)=\psi(t)$ follows 
with the same arguments as used in Theorem~\ref{thm:classical}. 
The presentation of the adjoint formally holds also for $\psi \in L^2(0,T)$, which can be proved by a limiting process.
The adjoint problem then has to be understood in a generalized sense.
\end{remark}

%
%
%
%
%

\section{Remarks about uniqueness and the approximate source condition} \label{sec:hyp}

We now collect some comments about the uniqueness hypothesis \eqref{eq:unique} and the approximate source condition \eqref{eq:source}. Our considerations are based on the fact that the nonlinear inverse problem is actually close to a linear inverse problem provided that the experimental setup is chosen appropriately. 
We will only sketch the main arguments here with the aim to illustrate the plausibility of these assumptions and to 
explain what results can be expected in the numerical experiments presented later on.

\subsection{Reconstruction for a stationary experiment}
Let the boundary data \eqref{eq:sys3} be chosen such that 
$g_0(t)=g_1(t)=\bar g \in \RR$ for $t \ge t_0$. 
By the energy estimates of \cite{GattiPata06}, which are derived for an equivalent problem in second order form \eqref{eq:second} there, one can show that the solution 
$(p(t),u(t))$ of the system \eqref{eq:sys1}--\eqref{eq:sys4} converges exponentially fast to a steady state 
$(\bar p,\bar u)$, which is the unique solution of 
\begin{align}
\dx \bar u &= 0, \quad x \in (0,1), \label{eq:stat1}\\
\dx \bar p + a(\bar u) &= 0, \quad x \in (0,1), \label{eq:stat2} 
\end{align}
with boundary condition $\bar u(0)=\bar u(1)=\bar g$.
From equation \eqref{eq:stat1}, we deduce that the steady state $\bar u$ is constant, 
and upon integration of \eqref{eq:stat2}, we obtain 
\begin{align} \label{eq:statsol}
a(\bar g) = a(\bar u) = \int_0^1 a(\bar u) dx = -\int_0^1 \dx \bar p(x) dx = \bar p(0)-\bar p(1) = \tpb.
\end{align}
The value $a(\bar g)$ can thus be determined by a stationary experiment. 
As a consequence, the friction law $a(\cdot)$ could in principle be determined from an infinite number of 
stationary experiments. 
We will next investigate the inverse problem for these \emph{stationary experiments} in detail.
In a second step, we then use these results for the analysis of the inverse problem for the instationary experiments 
that are our main focus.

\subsection{A linear inverse problem for a series of stationary experiments}

Let us fix a smooth and monotonic function $g : [0,T] \to [\umin,\umax]$ and denote by $\tpb(t)=a(g(t))$ the pressure difference obtained from the stationary system \eqref{eq:stat1}--\eqref{eq:stat2} with boundary flux $\bar g=g(t)$.
The forward operator for a sequence of stationary experiments is then given by
\begin{align} \label{eq:statop}
K : H^2(\umin,\umax) \to L^2(0,T), \qquad a \mapsto a(g(\cdot)),
\end{align}
and the corresponding inverse problem with exact data reads
\begin{align} \label{eq:statinv}
(K a)(t) = a(g(t)) = \tpb(t), \qquad t \in [0,T].
\end{align}
This problem is linear and its solution is given by the simple formula \eqref{eq:statsol} 
with $\bar g$ and $\tpb$ replaced by $g(t)$ and $\tpb(t)$ accordingly. 
%
From this representation, it follows that 
\begin{align*}
\|a - \tilde a\|^2_{L^2(\umin,\umax)} 
&= \int_0^T |a(g(t)) - \tilde a(g(t))|^2 |g'(t)| dt 
 \le C \|K a - K \tilde a\|_{L^2(0,T)}^2,
\end{align*}
where we assumed that $|g'(t)| \le C$ for all $t$.
Using the uniform bounds iny assumption (A1), 
embedding, and interpolation, one can further deduce that 
\begin{align} \label{eq:stathoelder}
\|a - \tilde a\|_{H^2(\umin,\umax)} \le C_\gamma \|K a-K\tilde a\|_{L^2(0,T)}^\gamma.
\end{align}
This shows the Hölder stability of the inverse problem \eqref{eq:statinv} for stationary experiments. 
As a next step, we will now extend these results to the instationary case by a perturbation argument 
as proposed in \cite{EggerPietschmannSchlottbom15} for a related inverse heat conduction problem.

\subsection{Approximate stability for the instationary inverse problem}

If the variation of the boundary data $g(t)$ with respect to time is sufficiently small, 
then from the exponential stability estimates of \cite{GattiPata06}, one may deduce that 
\begin{align} \label{eq:eps}
 \|p(t)-\bar p(t)\|_{H^1} + \|u(t)-\bar u(t)\|_{H^1} \le \eps.
\end{align}
Hence the solution $(p(t),u(t))$ is always close to the stationary state $(\bar p(t),\bar u(t))$ 
with the corresponding boundary data $\bar u(0,t)=\bar u(1,t)=g(t)$. 
Using $\triangle p(t) = p(0,t) - p(1,t) = -\int_0^1 \dx p(y,t) dy$ and the Cauchy-Schwarz inequality leads to
\begin{align} \label{eq:est}
|\tp(t) - a(g(t))| = |\tp(t) - \tpb(t)| \le \|p(t)-\bar p(t)\|_{H^1(0,1)} \le \eps. 
\end{align}
From the definition of the nonlinear and the linear forward operators, we deduce that
\begin{align} \label{eq:estOps}
F(a) = K a + O(\eps).
\end{align}
\begin{remark}
As indicated above, the error $\eps$ can be made arbitrarily small by a proper design of the experiment, i.e., by 
slow variation of the boundary data $g(t)$.
The term $O(\eps)$ can therefore be considered as an additional measurement error,
and thus the parameter $a$ can be determined approximately with the formula \eqref{eq:statsol} 
for the stationary experiments.
As a consequence of the stability of the linear inverse problem, we further obtain 
\begin{align} \label{eq:hoelder}
\|a-\tilde a\|_{H^2(\umin,\umax)} \le C'_\gamma \|F(a)-F(\tilde a)\|_{L^2(0,T)}^\gamma + C''_\gamma \eps^\gamma. 
\end{align}
In summary, we may thus expect that the identification from the nonlinear experiments is stable and unique, 
provided that the experimental setup is chosen appropriately. 
\end{remark}

\subsection{The approximate source condition}

With the aid of the stability results in \cite{GattiPata06} and similar reasoning as above, 
one can show that the linearized operator satisfies
\begin{align*}
F'(a) h = K h + O(\eps \|h\|_{H^2(\umin,\umax)}).
\end{align*}
A similar expansion is then also valid for the adjoint operator, namely
\begin{align*}
F'(a)^* w  = K^*w + O(\eps \|w\|_{L^2(0,T)}).
\end{align*}
This follows since $L=F'(a)-K$ is linear and bounded by a multiple of $\eps$, 
and so is the adjoint $L^*=F'(a)^*-K^*$.
In order to verify the approximate source condition \eqref{eq:source}, it thus suffices to 
consider the condition $z = K^* w$ for the linear problem. 
From the explicit respresentation \eqref{eq:statop} of the operator $K$ this can be translated directly to a smoothness condition on $z$  in terms of weighted Sobolev spaces and some boundary conditions; we refer to \cite{EggerPietschmannSchlottbom14} for a detailed derivation in a similar context. 

\begin{remark}
The observations made in this section can be summarized as follows:

(i) If the true parameter $a$ is sufficiently smooth, and if the boundary data are varied sufficiently slowly ($\eps$ small), 
such that the instationary solution at time $t$ is close to the steady state corresponding to the boundary data $g(t)$,
then the parameter can be identified stably with the simple formula for the linear inverse problem. 
The same stable reconstructions will also be obtained with Tikhonov regularization \eqref{eq:tikhonov}. 

(ii) For increasing $\eps$, the approximation \eqref{eq:estOps} of the nonlinear problem by the linear problem deteriorates. 
In this case, the reconstruction by the simple formula \eqref{eq:statsol} will get worse while the solutions obtained 
by Tikhonov regularization for the instationary problem can be expected to still yield good and stable reconstructions.
\end{remark}

\begin{remark}
Our reasoning here was based on the \emph{approximate stability estimate} \eqref{eq:hoelder} that is inherited from the 
satationary problem by a perturbation argument. 
A related analysis of Tikhonov regularization under \emph{exact} conditional stability assumptions
can be found \cite{ChengYamamoto00,HofmannYamamoto10} together with some applications.
\end{remark}

\section{Numerical tests} \label{sec:num}

For illustration of our theoretical considerations discussed in the previous section, 
let us we present some numerical results which provide additional evidence for the 
uniqueness and stability of the inverse problem.

\subsection{Discretization of the state equations}

For the space discretization of state system \eqref{eq:sys1}--\eqref{eq:sys4}, 
we utilize a mixed finite element method based on a weak formulation of the problem. 
The pressure $p$ and the velocity $u$ are approximated with continuous piecewise linear 
and discontinuous piecewise constant finite elements, respectively. For the time discretization, we employ a one step scheme in which the differential terms are treated implicitly and the nonlinear damping term is integrated explicitly. 
A single time step of the resulting numerical scheme then has the form 
\begin{align*}
\tfrac{1}{\tau} (p^{n+1}_h,q_h) - (u_h^{n+1},\dx q_h) &= \tfrac{1}{\tau} (p_h^n,q_h) + g_0^{n+1} q_h(0) - g_1^{n+1} q_h(1), \\ 
\tfrac{1}{\tau} (u_h^{n+1},v_h) + (\dx p_h^{n+1},v_h) &= \tfrac{1}{\tau} (u_h^n,v_h) - (a(u_h^n),v_h), 
\end{align*}
for all test functions $q_h \in P_1(T_h) \cap C[0,1]$ and $v_h \in P_0(T_h)$.
Here $T_h$ is the mesh of the interval $(0,1)$, $P_k(T_h)$ denotes the space of piecewise polynomials on $T_h$, $\tau>0$ is the time-step, and $g_i^n=g_i(t^n)$ are the boundary fluxes at time $t^{n}=n \tau$. 
The functions $(p_h^n,u_h^n)$ serve as approximations for the solutions $(p_h(t^n),u_h(t^n))$
at the discrete time steps.
Similar schemes are used to approximate the sensitivity system \eqref{eq:sen1}--\eqref{eq:sen34} and the adjoint problem \eqref{eq:adj}--\eqref{eq:adj3} in a consistent manner.
The spatial and temporal mesh size were chosen so small such that 
approximation errors due to the discretization can be neglected;
this was verified by repeating the tests with different discretization parameters.

\subsection{Approximation of the parameter}

The parameter function $a(\cdot)$ was approximated by cubic interpolating splines over a uniform grid of the interval $[\umin,\umax]$. 
The splines were parametrized by the interpolation conditions $s(u_i)=s_i$, $i=0,\ldots,m$ and knot-a-knot conditions was used to obtain a unique representation. To simplify the implementation, the $L^2$, $H^1$, and $H^2$ norm in the parameter space were approximated by difference operators acting directly on the interpolation points $s_i$, $i=0,\ldots,m$. 
To ensure mesh independence, the tests were repeated for different 
numbers $m$ of interpolation points.

\subsection{Minimization of the Tikhonov functional}

For minimization of the Tikhonov functional \eqref{eq:tikhonov}, we utilized a projected iteratively regularized Gau\ss-Newton method with regularization parameters $\alpha^n = c q^n$, $q<1$. 
The bounds in assumption (A1) for the parameters were satisfied automatically for all iterates in our tests such that  the projection step was never carried out. 
The iteration was stopped by a discrepancy principle, i.e., when $\|F(a^n) - h^\delta\| \le 1.5 \delta$ was valid the first time. 
The regularization parameter $\alpha^n$ of the last step was interpreted as the regularization parameter $\alpha$ of the Tikhonov functional \eqref{eq:tikhonov}.
We refer to \cite{EggerSchlottbom11} for details concerning such a strategy for the iterative minimization of the Tikhonov functional.
The discretizations of the derivative and adjoint operators $F'(a)$ and $F'(a)^*$ were implemented consistently, such that $(F'(a) h, \psi) = (h, F'(a)^* \psi)$ holds exactly also on the discrete level. 
The linear systems of the Gau\ss-Newton method were then solved by a preconditioned conjugate gradient algorithm. 

\subsection{Setup of the test problem}

As true damping parameter, we used the function
\begin{align} \label{eq:adag}
a^\dag(u) = u \sqrt{1+u^2} . 
\end{align}
The asymptotic behaviour here is 
$a(u) \approx u$ for $|u| \ll 1$ and $a(u) \approx u |u|$ for $|u| \gg 1$, 
which corresponds to the expected behaviour of the friction forces in pipes \cite{LandauLifshitz6}.
Restricted to any bounded interval $[\umin,\umax]$, the function $a^\dag$ satisfies the assumptions (A1).

For our numerical tests, we used the initial data $u_0 \equiv 0$, $p_0 \equiv 1$, 
and we chose 
\begin{align} \label{eq:g}
g_0(t)=g_1(t)=g(t)=2 \sin(\tfrac{\pi}{2T} t)^2
\end{align}
as boundary fluxes. 
A variation of the time horizon $T$ thus allows us to tune the speed of variation in the boundary data, 
while keeping the interval $[\umin,\umax]$ of fluxes that arise at the boundary fixed. 

\subsection{Simulation of measurement data}
The boundary data $g(t;T)$ and the resulting pressure drops $\tp(t;T)$ across the pipe resulting are displayed 
in Figure~\ref{fig:1} for different choices of $T$. 
For comparison, we also display the pressure drop $\tpb$ obtained with the 
linear forward model.

\begin{figure}[ht!]
\medskip
\includegraphics[height=3.2cm]{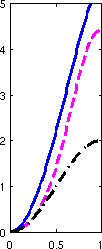} \hspace*{0.75cm}
\includegraphics[height=3.2cm]{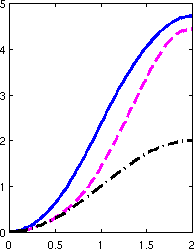} \hspace*{0.75cm}
\includegraphics[height=3.2cm]{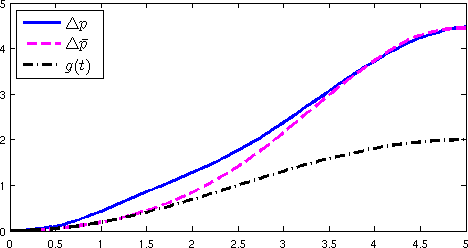} \\[0.5cm]
\includegraphics[height=3.2cm]{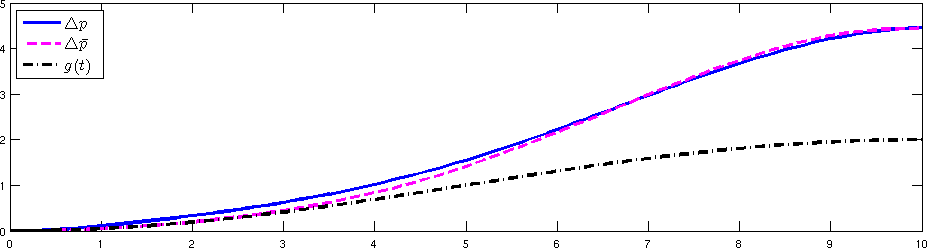}
\caption{Boundary flux $g(t)$ and pressure drops $\tp(t)$ and $\tpb(t)$ for the instationary and the linearized model for time horizon $T=1,2,5,10$. \label{fig:1}}
\end{figure}

The following observations can be made: 
For small values of $T$, the pressure drop $\tp$ varies rapidly all over the time interval $[0,T]$ and 
therefore deviates strongly from the pressure drop $\tpb$ of the linearized model corresponding to stationary 
experiments.
In contrast, the pressure drop $\tp$ is close to that of the linearized model on the whole time interval $[0,T]$, when $T$ is large and therefore the variation in the boundary data $g(t)$ is small. 
As expected from \eqref{eq:estOps}, the difference between $\tp$ and $\tpb$ becomes smaller when $T$ is increased.
A proper choice of the parameter $T$ thus allows us to tune our experimental setup 
and to verify the conclusions obtained in Section~\ref{sec:hyp}.

\subsection{Convergence to steady state}

We next demonstrate in more detail that the solution 
$(p(t),u(t))$ of the instationary problem is close to the steady states $(\bar p(t),\bar u(t))$ for 
boundary data $\bar u(0)=\bar u(1)=g(t)$, provided that $g(t)$ varies sufficiently slowly; cf \eqref{eq:eps}.
In Table~\ref{tab:1}, we list the errors 
\begin{align*}
e(T) := \max_{0 \le t \le T} \|p(t;T)-\bar p(t;T)\|_{L^2} + \|u(t;T)-\bar u(t;T)\|_{L^2} 
\end{align*}
between the instationary and the corresponding stationary states 
for different values of $T$ in the definition of the boundary data $g(t;T)$.
In addition, we also display the difference 
\begin{align*}
d(T)=\max_{0 \le t \le T} |\tp(t) - \tpb(t)| 
\end{align*}
in the measurements corresponding to the nonlinear and the linearized model.
\begin{table}[ht!]
\centering
\small
\begin{tabular}{c||c|c|c|c|c|c}
 $T$    & $1$     & $2$     & $5$     & $10$    & $20$    & $50$ \\
\hline
 $e(T)$ & $1.016$ & $0.647$ & $0.207$ & $0.105$ & $0.054$ & $0.022$ \\
\hline
 $d(T)$ & $1.044$ & $1.030$ & $0.479$ & $0.225$ & $0.114$ & $0.045$
\end{tabular}
\medskip

\caption{Error $e(T)$ between instationary and stationary solution and difference $d(T)$ in the corresponding 
measurements.\label{tab:1}} 
\end{table}
The speed of variation in the boundary data decreases when $T$ becomes larger, 
and we thus expect a monotonic decrease of the distance $e(T)$ to steady state 
with increasing time horizon. The same can be expected for the error $d(T)$
in the measurements. This is exactly the behaviour that we observe in our numerical tests. 

\subsection{Reconstructions for nonlinear and linearized model}

Let us now turn to the inverse problem
and compare the reconstructions for the nonlinear inverse problem obtained by Tikhonov regularization with that computed by the simple formula \eqref{eq:statsol} for the linearized inverse problem corresponding to stationary experiments. 
The data for these tests are generated by simulation as explained before, and then perturbed with random noise such that $\delta=0.001$.
Since the noise level is rather small, the data perturbations do not have any visual effect on the reconstructions here;
see also Figure~\ref{fig:3} below.

In Figure~\ref{fig:2}, we display the corresponding results for measurements $h=\tp(\cdot;T)$ obtained for different time horizons $T$ in the definition of the boundary data $g(\cdot;T)$. 

\begin{figure}[ht!]
\includegraphics[height=5cm]{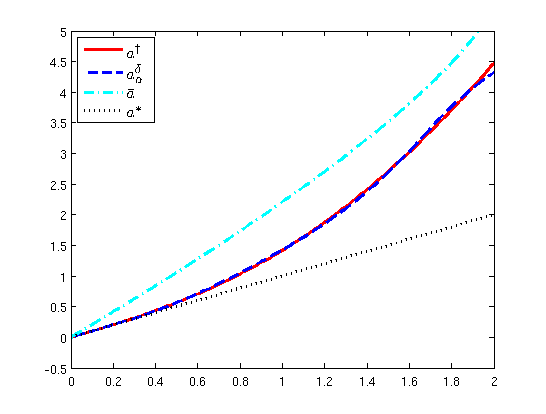} 
\includegraphics[height=5cm]{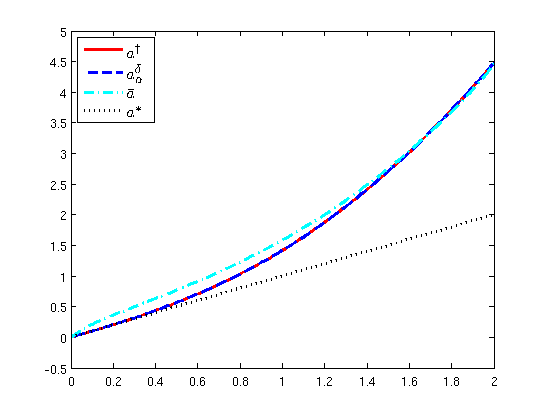} \\
\includegraphics[height=5cm]{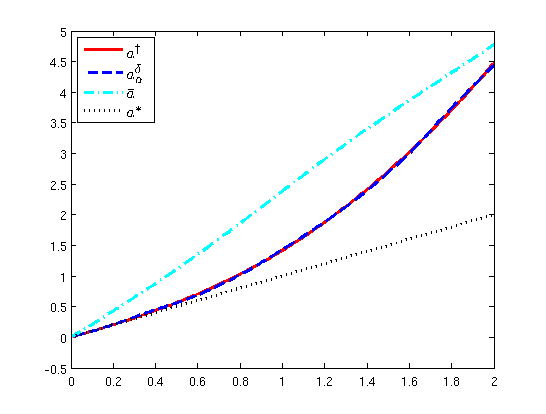} 
\includegraphics[height=5cm]{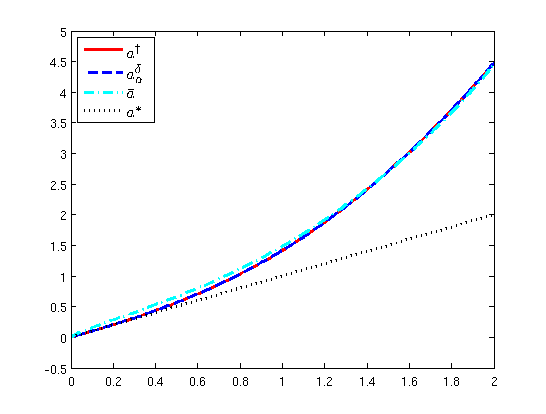} \\
\includegraphics[height=5cm]{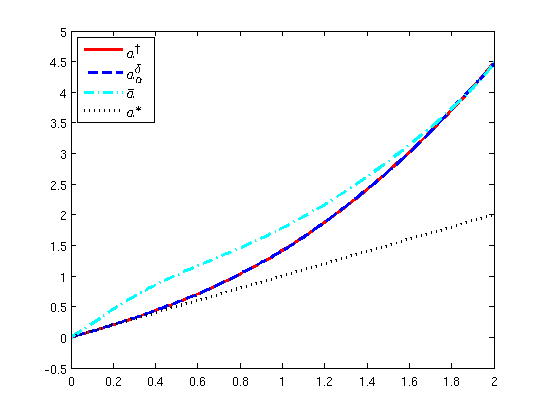} 
\includegraphics[height=5cm]{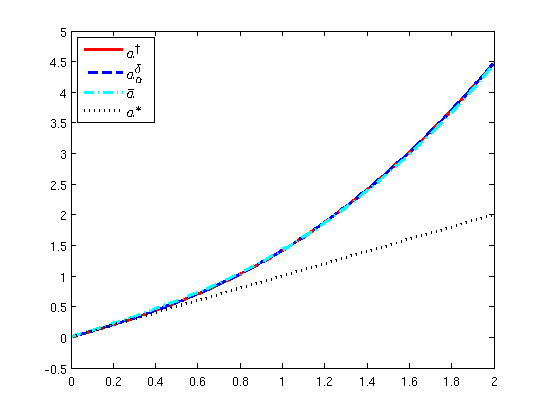}
\caption{True parameter $a^\dag$, reconstruction $a_\alpha^\delta$ obtained by Tikhonov regularization with initial guess $a^*$, and result $\bar a$ obtained by formula \eqref{eq:statsol}. 
The data $h^\delta$ are perturbed by random noise of size $\delta=0.001$.
The images correspond to time horizons $T=1,2,5$ (left) and $T=10,20,50$ (right).\label{fig:2}} 
\end{figure}

As can be seen from the plots, the reconstruction with Tikhonov regularization works well in all test cases. 
The results obtained with the simple formula \eqref{eq:statsol} however show some systematic deviations due to model errors, 
which however become smaller when increasing $T$. 
Recall that for large $T$, the speed of variation in the boundary fluxes $g(t;T)$ is small, 
so that the system is close to steady state on the whole interval $[0,T]$. 
The convergence of the  reconstruction $\bar a$ for the linearized problem towards the true solution $a^\dag$
with increasing $T$ is thus in perfect agreement with our considerations in Sections~\ref{sec:hyp}.

\subsection{Convergence and convergence rates}

In a last sequence of tests, we investigate the stability and accuracy of the 
reconstructions $a_\alpha^\delta$ obtained with Tikhonov regularization in the presence of data noise. 
Approximations for the minimizers $a_\alpha^\delta$ are computed numerically via the projected 
iteratively regularized Gau\ss-Newton method as outlined above. 
The iteration is stopped according to the discrepancy principle.
Table~\ref{tab:2} displays the reconstruction errors for different time horizons $T$ and different noise levels $\delta$. 

\begin{table}[ht!]
\centering
\small
\begin{tabular}{c||c|c|c|c|c|c}
$\delta \backslash T$   
          & $1$      & $2$      & $5$      & $10$     & $20$     & $50$ \\
\hline
\hline
$0.10000$ & $0.8504$ & $0.3712$ & $0.1027$ & $0.0417$ & $0.0324$ & $0.0092$ \\
\hline
$0.05000$ & $0.6243$ & $0.2742$ & $0.0706$ & $0.0239$ & $0.0081$ & $0.0055$ \\
\hline
$0.02500$ & $0.3911$ & $0.1616$ & $0.0496$ & $0.0096$ & $0.0066$ & $0.0032$ \\
\hline
$0.01250$ & $0.2264$ & $0.1050$ & $0.0355$ & $0.0065$ & $0.0024$ & $0.0019$ \\
\hline
$0.00625$ & $0.1505$ & $0.0630$ & $0.0316$ & $0.0030$ & $0.0015$ & $0.0012$ 
\end{tabular}
\medskip

\caption{Reconstruction error $\|a_\alpha^\delta - a^\dag\|_{L^2(\umin,\umax)}$ for Tikhonov regularization for different noise levels $\delta$ and various time horizons $T$. \label{tab:2}}
\end{table}

Convergence is observed for all experimental setups, but the absolut errors decrease 
monotonically with increasing time horizon $T$, which is partly explained by 
our considerations in Section~\ref{sec:hyp}.
The reconstructions for time horizon $T=2$, corresponding to the third column of Table~\ref{tab:2},
are depicted in Figure~\ref{fig:3}; also compare with Figure~\ref{fig:2}. 

\begin{figure}[ht!]
\includegraphics[height=5cm]{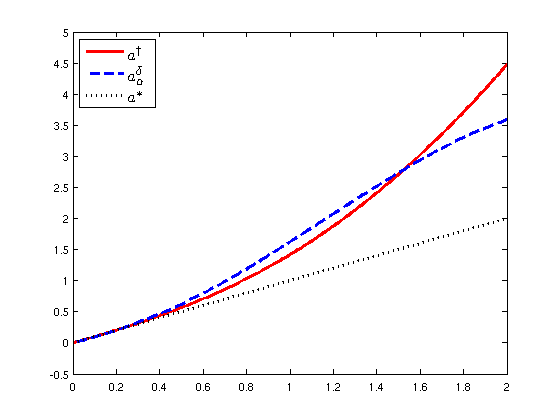} 
\includegraphics[height=5cm]{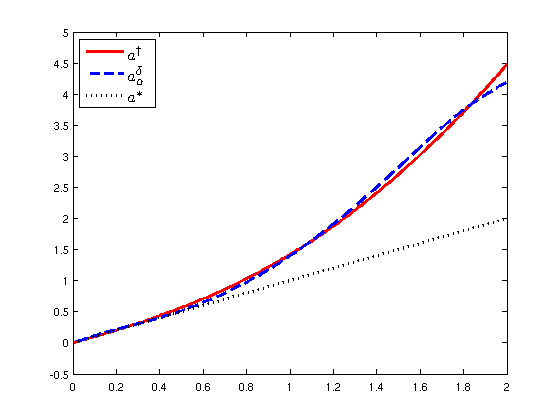} \\
\includegraphics[height=5cm]{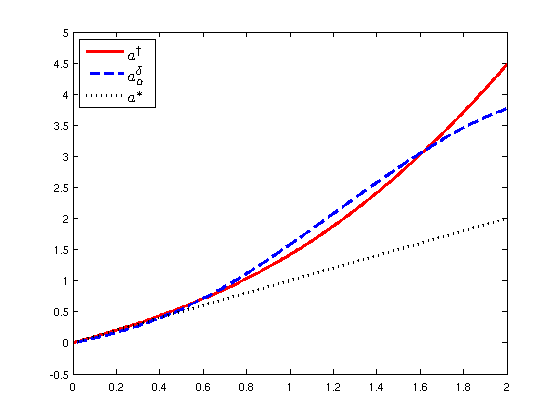} 
\includegraphics[height=5cm]{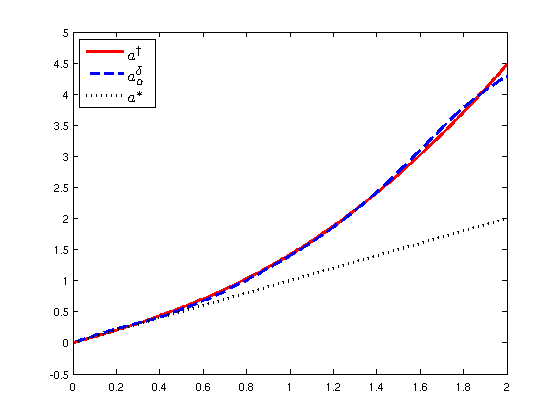} \\
\includegraphics[height=5cm]{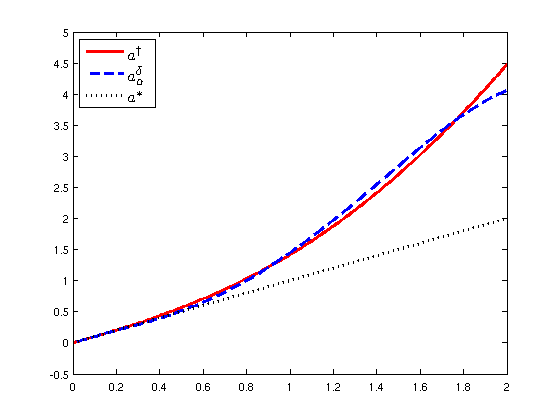} 
\includegraphics[height=5cm]{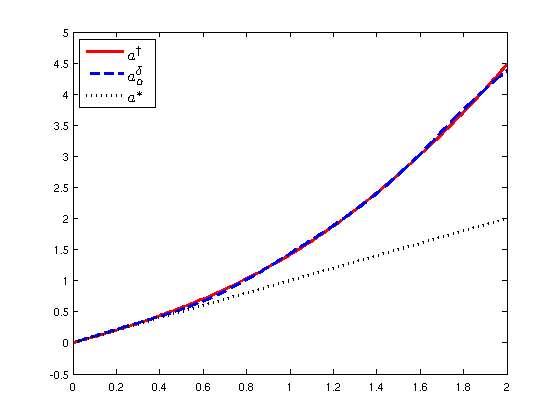}
\caption{True parameter $a^\dag$, reconstruction $a_\alpha^\delta$ obtained by Tikhonov regularization, and initial guess $a^*$ for time horizon $T=2$ and noise levels $\delta=0.1,0.05,0.025$ (left) and $\delta=0.0125,0.00625,0.003125$ (right).\label{fig:3}} 
\end{figure}

Note that already for a small time horizon $T=2$ and large noise level $\delta$ of several percent, one can obtain good reconstructions of the damping profile. For larger time horizon or smaller noise levels, the reconstruction $a_\alpha^\delta$ visually coincides completely with the true solution $a^\dag$.
This is in good agreement with our considerations in Section~\ref{sec:hyp}.

\section{Discussion} \label{sec:sum}

In this paper, we investigated the identification of a nonlinear damping law in a 
semilinear hyperbolic system from additional boundary measurements. Uniqueness and 
stability of the reconstructions obtained by Tikhonov regularization was observed
in all numerical tests. This behaviour could be explained theoretically by considering 
the nonlinear inverse problem as a perturbation of a nearby linear problem, for which 
uniqueness and stability can be proven rigorously.

In the coefficient inverse problem under investigation, the distance to the approximating 
linearization could be chosen freely by a proper experimental setup. A similar argument was
already used in \cite{EggerPietschmannSchlottbom15} for the identification of a nonlinear 
diffusion coefficient in a quasi-linear heat equation. 
The general strategy might however be useful in a more general context and for many other applications.

Based on the uniqueness and stability of the linearized inverse problem, we could obtain 
stability results for the nonlinear problem \emph{up to perturbations}; see 
Section~\ref{sec:hyp} for details. Such a concept might be useful as well for the convergence 
analysis of other regularization methods and more general inverse problems.

In all numerical tests we observed global convergence of an iterative method for the minimization 
of the Tikhonov functional. Since the minimizer is unique for the linearized problem, such
a behaviour seems not too surprising. At the moment, we can however not give a rigorous 
explanation of that fact. Let us note however, that H\"older stability of the inverse 
problem can be employed to prove convergence and convergence rates for Tikhonov regularization \cite{ChengHofmannLu14,ChengYamamoto00} and also global convergence of iterative regularization methods 
\cite{deHoopQiuScherzer12} without further assumptions. 
An extension of these results to inverse problems satisfying \emph{approximated stability conditions}, 
as the one considered here, might be possible.

\section*{Acknowledgements}
The authors are grateful for financial support by the German Research Foundation (DFG) via grants IRTG~1529, GSC~233, and TRR~154.


\appendix
\renewcommand\thesection{A}
\section*{Appendix}

We consider semilinear evolution problems of the abstract form
\begin{align} \label{eq:cauchy}
y'(t) + A y(t) = f(t,y(t)), \ t>0, \qquad y(0)=y_0,
\end{align}
where $X$ is a Banach space and $A:D(A) \subset X \to X$ is the generator of a $C^0$-semigroup 
of contractions on $X$ denoted by $\{e^{-At}\}_{t \ge 0}$. 
The analysis of such systems is well-established; see e.g. \cite[Ch~6.2]{Pazy83}.
We recall some basic results here for later reference.
\begin{lemma}[Mild solution]\label{lem:mild}
Let $f : [0,T] \times X \to X$ be continuous on $[0,T]$ with 
\begin{align*}
\|f(t)\|_X \le C_1 
\end{align*}
and assume that $f$ is uniformly Lipschitz-continuous with respect to $y$, i.e., 
 \begin{align*} 
   \|f(t,y) - f(t,z)\|_X \le L \|y-z\|_X \qquad \text{for all } y,z \in X \text{ and } t \in [0,T].
 \end{align*}
Then for any $y_0 \in X$ the Cauchy problem \eqref{eq:cauchy} has a unique mild solution $y \in C([0,T];X)$ defined by the \emph{variation-of-constants formula}
\begin{align} \label{eq:integral}
 y(t) = e^{-At} y_0 + \int_0^t e^{-A(t-s)} f(s,y(s)) ds.
\end{align}
The solution is bounded by $\|y\|_{C([0,T];X)} =\max_{0 \le t \le T} \|y(t)\|_X \le C'(T,C_1,L,\|y_0\|_X)$.  
\end{lemma}
\begin{proof}
The existence of a unique mild solution is stated in \cite[Ch~6, Thm~1.2]{Pazy83}.
Since the semigroup is contractive, we have $\|e^{-As}\|_X \le 1$ for $s \ge 0$, and thus 
\begin{align*}
\|y(t)\|_X \le \|y_0\|_X + \int_0^t \|f(s,y(s))\|_X ds.
\end{align*}
The assertion now follows by using the bound for $f$ and the Gronwall inequality~\cite{Gronwall19}.
\end{proof}

The existence of classical solutions can be guaranteed under stronger assumptions.
\begin{lemma}[Classical solution] \label{lem:classical} $ $\\
Assume in addition that $f:[0,T] \times X \to X$ is continuously differentiable
with
 \begin{align*} 
   \|f_t(t,y)\|_X \le C_2 \qquad \text{and} \qquad \|f_y(t,y)\|_{X \to X} \le C_3 
 \end{align*}
for all $y \in X$ and $t \in [0,T]$. 
Then for any $y_0 \in D(A)$
the mild solution of Lemma~\ref{lem:mild} is a classical solution of \eqref{eq:cauchy}, i.e.,  $y \in C^{1}([0,T];X) \cap C([0,T];D(A))$, and 
\begin{align*}
 \|y\|_{C^{1}([0,T];X) \cap C([0,T];D(A))} \le C''(T,C_1,C_2,C_3,\|y_0\|_X+\|A y_0\|_X). 
\end{align*}
\end{lemma}
The norm here is given by $\|y\|_{C^{1}([0,T];X) \cap C([0,T];D(A))} = \max\limits_{0 \le t \le T} \|y(t)\|_X + \|y'(t)\|_X + \|A y(t)\|_X$.
\begin{proof}
The regularity statement can be found in \cite[Ch~6, Thm~1.5]{Pazy83}. 
By formal derivation of \eqref{eq:cauchy} with respect to time, one obtains the linear system
\begin{align*} 
z'(t) + A z(t) = f_t(t,y(t)) + f_y(t,y(t)) z(t), \qquad z(0)=z_0, 
\end{align*}
with $z(t)=y'(t)$ and $z_0=y'(0)=f(0,y_0)-A y_0$. 
The previous lemma provides existence of a mild solution $z$ and the bound $\|y'(t)\|_X=\|z(t)\|_{X} \le C$. The assertion then follows by making use of 
equation \eqref{eq:cauchy}.
\end{proof}


The following result provides an alternative way for proving existence of a classical solution.
Let us define $Y=D(A)=\{y \in X : A y \in X\}$, which is again a Banach space when equipped with the norm 
$\|y\|_Y^2 = \|y\|_X^2 + \|A y\|_X^2$. 

\begin{lemma}[Local classical solution] \label{lem:classical2}
Let $f : [0,T] \times Y \to Y$ be continuous with
 \begin{align} \label{eq:FF1}
    \|f(t,0)\|_Y \le C_4 \qquad \text{for all } t \in [0,T]
 \end{align}
and assume that $f$ is locally Lipschitz-continuous with respect to $y$ uniformly in $t$, 
i.e., for all $C>0$ and $y,z \in Y$ with $\|y\|_Y,\|z\|_Y \le C$ there exists $L_C<\infty$ such that 
\begin{align} \label{eq:FF2}
   \|f(t,y) - f(t,z)\|_Y \le L_C \|y-z\|_Y \qquad t \in [0,T].
 \end{align}
Then for any $y_0 \in Y$ problem \eqref{eq:cauchy} has a unique 
local classical solution $y \in C([0,T'];Y)$ and 
$\|y\|_{C([0,T'];Y)} \le C'''$ with $T'$ and $C'''$ only depending on $C$, $L_C$, and $C_4$.

\end{lemma}
\begin{proof}
The assertion is a consequence of Lemma~\ref{lem:mild} and the fact, 
that $\{e^{-At}\}_{t \ge 0}$  also defines a contraction semigroup on $Y$; see  \cite[Ch~6, Thm~1.7]{Pazy83}
for details.
\end{proof}

The following result allows to deduce global existence and a-priori bounds.
\begin{lemma}[Global classical solution] \label{lem:classical3} %
In addition assume that
\begin{align*} 
\|\frac{d}{dt} f(t,y(t))  \|_X \le C_1 + C_2 \|\dt y(t)\|_X + C_3 \|A y(t)\|_X,
\end{align*}
for any $y \in C^1([0,T];X) \cap C([0,T];D(A))$. 
Then the solution of Lemma~\ref{lem:classical2}, 
can be extended to the whole interval $[0,T]$ and bounded as in Lemma~\ref{lem:classical}.
\end{lemma}
\begin{proof}
By Lemma~\ref{lem:classical2}, a local classical solution $y$ exists. 
Formal differentiation of \eqref{eq:cauchy} shows that 
the derivative $z(t)=y'(t)$ satisfies
\begin{align*}
z'(t) + A z(t) = g(t,z(t)), \qquad z(0)=z_0  
\end{align*}
with $z_0 = y'(0)$ and $g(t,z(t))=\ddt f(t,y(t))$.
Since $y$ is a classical solution, we see that a mild solution $z=\dt y$ exists.
Using the variation of constants formula \eqref{eq:integral} for $z$ yields
\begin{align*}
\|z(t)\|_X \le \|e^{-At}\|_X \|z_0\|_X + \int_0^t \|e^{-A(t-s)}\|_X \|\ddt f(s,y(s))\|_X ds.
\end{align*}
Hence the classical solution $y$ of \eqref{eq:cauchy} is uniformly which implies the assertion.
\end{proof}

\end{document}